\documentclass[a4paper,10pt]{article}

\usepackage{amsmath}  
\usepackage{amssymb}            
\usepackage{amsxtra}
\usepackage{amscd}
\usepackage[mathscr]{euscript}

\everymath{\displaystyle} 

\newtheorem{thm}{Theorem}[subsection]
\newtheorem{prop}[thm]{Proposition}

\newcommand{\npmod}[1]{\!\!\pmod{#1}}
\newcommand{\nnpmod}[1]{\!\!\!\!\pmod{#1}}

\newenvironment{proof}{\par\noindent{\bf[Proof]}}%
                      {$\blacksquare$\noindent\par\vspace{0.5\baselineskip}}
                      {$\blacksquare$\par\noindent}

\font\b=cmr10 scaled \magstep4
\def\bigzerol{\smash{\hbox{\b 0}}}
\def\bigzerou{\smash{\lower0.7ex\hbox{\b 0}}}
\def\bigastl{\smash{\lower0.7ex\hbox{\b *}}}
\def\bigastu{\smash{\lower2.7ex\hbox{\b *}}}

\def\addots{\mathinner
    {\mkern1mu\raise1pt\hbox{.}\mkern2mu
    \raise4pt\hbox{.}\mkern2mu\raise7pt\vbox{\kern7pt\hbox{.}}\mkern1mu}}

\makeatletter
\renewcommand{\subsection}{\@startsection%
  {subsection}%
  {2}%
  {0mm}%
  {\baselineskip}%
  {-0.2\parindent}%
  {\normalfont\normalsize\upshape\bfseries}}%
\def\@dotsep{1.5}
\def\@pnumwidth{1em}
\makeatother

\title{On generic supercuspidal representations of $Sp_{2n}$} 
\author{Koichi Takase
        \thanks{The author is partially supported by 
                    JSPS KAKENHI Grant Number JP 16K05053}}
\date{}

\begin{document}   


%
%

\maketitle

\begin{abstract}
We will construct a family of irreducible generic supercuspidal
representations of the symplectic groups over
non-archimedian local field $F$ of odd residual characteristic. The
supercuspidal representations are compactly induced from irreducible
representations of the hyperspecial compact subgroup which are
inflated from irreducible representations of finite symplectic groups 
over the finite quotient ring of the integer 
ring of $F$ modulo high powers of the prime element. 

\end{abstract}

\section{Introduction}
\label{sec:introduction}

The supercuspidal representations of a classical group $G$ (that is a
unitary, symplectic or special orthogonal group) over a non-arcimedian
local field $F$ are realized as
compactly  induced representations $\text{\rm ind}_J^{G(F)}\delta$
where $\delta$ is an irreducible finite dimensional complex
representation of an open compact subgroup $J$ of $G(F)$ (see 
\cite{Stevens08}). 

If $G$ is quasi-split and unramified over $F$, then $G$ has a smooth model
over the integer ring $O_F$ of $F$. In this case, 
\cite{DeBackerReeder2010} shows that the supercuspidal representation 
$\text{\rm ind}_J^{G(F)}\delta$ is generic if and only if $J$ is
hyperspecial. We will consider the case $J=G(O_F)$. 

The irreducible complex
representations of $G(O_F)$ come from those of the finite group 
$G(O_F/\frak{p}^r)$ ($\frak{p}$ is the maximal ideal of $O_F$) via the
canonical surjection $G(O_F)\to G(O_F/\frak{p}^r)$. If $r=1$, we have 
well-developed theories on the representation of the finite reductive
group $G(\Bbb F)$ 
($\Bbb F=O_F/\frak{p}$ is the residue class field of $F$). 
The understanding of the cases $r>1$ is less complete. 
For example the representation theory of $GL_n(O_F/\frak{p}^r)$ with
$r>1$ is studied by  \cite{Shintani1968} and then \cite{Gerardin1972} to
construct supercuspidal representations of $GL_n(F)$, however their
descriptions of representations of $GL_n(O_F/\frak{p}^r)$ are rather
complicated and not suitable for detailed treatment or for
generalization to other classical groups.

Recently \cite{Takase2017} gives a uniform and quite explicit
description of the irreducible representations of $G(O_F/\frak{p}^r)$
($r>1$) associated with the regular adjoint orbit, under the
assumption of triviality of certain Schur multiplier. 
The Schur multiplier is trivial if the
characteristic polynomial of the adjoint orbit is irreducible modulo
$\frak{p}$. So starting from this explicit description of irreducible
complex representations $\delta$ of $G(O_F)$, we may be able to give a
simple proof of the fact that the compactly induced representation 
$\text{\rm ind}_{G(O_F)}^{G(F)}\delta$ is a generic supercuspidal 
representation. It also gives a good foundation for the further
detailed studies on the generic supercuspidal representations of
classical groups.

In this paper, we will study this procedure in detail for the case of
the symplectic group. Our main results are presented as Theorem 
\ref{th:main-result}. The rest of this paper is devoted to the
proofs. 


\section{Main results}
\label{sec:main-results}
\subsection[]{}
\label{susec:definition-of-symplectic-group}
Let $F$ be a non-dyadic non-archimedian local field. The integer ring
of $F$ is denoted by $O_F$ with the maximal ideal 
$\frak{p}$ generated by $\varpi$. The residue class field is denoted
by $\Bbb F=O_F/\frak{p}$. Let $\tau$ be a continuous unitary
additive character of $F$ such that 
$$
 \{x\in F\mid \tau(xO_F)=1\}=O_F.
$$
Let $G=Sp_{2n}$ be the symplectic group 
and $\frak{g}=\frak{sp}_{2n}$ its Lie algebra 
with respect to the symplectic form
$$
 J_n=\begin{bmatrix}
      0&I_n\\
     -I_n&0
     \end{bmatrix}\;\text{\rm where}\;
 I_n=\begin{bmatrix}
       &       &1\\
       &\addots& \\
      1&       &
     \end{bmatrix},
$$
that is to say 
$G$ (resp. $\frak{g}$) is a smooth $O_F$-group scheme 
(resp. affine $O_F$-scheme) such that
$$
 G(L)=\{g\in GL_{2n}(L)\mid gJ_n\,^tg=J_n\}
$$
$$
 (\text{\rm resp. 
  $\frak{g}(L)=\{X\in\frak{gl}_{2n}(L)\mid XJ_n+J_n\,^tX=0\}$})
$$
for any commutative $O_F$-algebra $L$. Let us denote by
$$
 B:\frak{g}\times\frak{g}\to\text{\rm Spec}(O_F[t])
$$
the trace form, that is $B(X,Y)=\text{\rm tr}(XY)$ for all 
$X,Y\in\frak{g}(L)$ with any commutative $O_F$-algebra $L$.

For any $O_F$-group subscheme $H\subset G$ and 
for any integers $0<r<l$, let
us denote by $H(\frak{p}^l/\frak{p}^r)$ (resp. $H(\frak{p}^r)$) 
the kernel of the canonical
group homomorphism $H(O_F/\frak{p}^r)\to H(O_F/\frak{p}^l)$ 
(resp. $H(O_F)\to H(O_F/\frak{p}^r)$). These group homomorphisms are 
surjective if $H$ is smooth over $O_F$. 

\subsection[]{}\label{subsec:main-result}
Fix an integer $r>1$ and put $r=l+l^{\prime}$ with the smallest 
integer $l$ such that $0<l^{\prime}\leq l$, in other word
$$
 l^{\prime}=\begin{cases}
             l&:\text{\rm if $r=2l$ is even,}\\
             l-1&:\text{\rm if $r=2l-1$ is odd.}
            \end{cases}
$$
Fix a $\beta\in\frak{g}(O_F)$ such that the reduction modulo
$\frak{p}$ of the characteristic polynomial $\chi_{\beta}(t)$ of 
$\beta\in\frak{gl}_{2n}(O_F)$ is irreducible over $\Bbb F$. 
Then the $F$-subalgebra $K=F[\beta]$ of the matrix algebra 
$M_{2n}(F)$ is an unramified field extension of $F$ of degree $2n$. 
The element $\tau\in\text{\rm Gal}(K/F)$ of order $2$ of 
the cyclic extension $K/F$ is given by 
$x^{\tau}=J_n\,^txJ_n^{-1}$. In particular $\beta^{\tau}=-\beta$. The
integer ring and the residue class field of $K$ are denoted by 
$O_K$ and $\Bbb K=O_K/\varpi O_K$ 
respectively. On the other hand $\beta$ is a regular
semisimple element of $\frak{g}(F)$ and its centralizer 
$T=Z_G(\beta)$ in $G$ is a maximal torus of $G$ which splits over
$K$. More precisely $T$ is a smooth $O_F$-group subscheme of $G$ such
that 
$$
 T(L)=\left\{g\in L[\overline{\beta}]^{\times}\mid
                            gJ_n\,^tgJ_n^{-1}=1_{2n}\right\}
$$
for any commutative $O_F$-algebra where 
$\overline{\beta}\in\frak{g}(L)$ is the image of 
$\beta\in\frak{g}(O_F)$ by the canonical homomorphism 
$\frak{g}(O_F)\to\frak{g}(L)$. In particular 
\begin{equation}
 T(O_F/\frak{p}^r)
 =\left\{\overline\varepsilon
         \in\left(O_K/\varpi^rO_K\right)^{\times}\mid
   \varepsilon\in O_K^{\times}\,\text{\rm s.t.}\,
   N_{K/K_+}(\varepsilon)=1\right\}
\label{eq:description-of-t-o-mod-p-r}
\end{equation}
where $K_+\subset K$ is the fixed subfield of 
$\tau\in\text{\rm Gal}(K/F)$. 

A group isomorphism 
$\frak{g}(O_F/\frak{p}^{l^{\prime}})\,\tilde{\to}\,
 G(\frak{p}^l/\frak{p}^r)$ is given by 
$$
 X\nnpmod{\frak{p}^{l^{\prime}}}\mapsto
 1_{2n}+\varpi^lX\nnpmod{\frak{p}^r}.
$$
Let us denote by $\psi_{\beta}$ the unitary character of 
$G(\frak{p}^l/\frak{p}^r)$ defined by
$$
 \psi_{\beta}\left(1_{2n}+\varpi^lX\nnpmod{\frak{p}^r}\right)
 =\tau\left(\varpi^{-l^{\prime}}B(X,\beta)\right).
$$
On $T(O_F/\frak{p}^r)\cap G(\frak{p}^l/\frak{p}^r)$, the character
$\psi_{\beta}$ is described by
$$
 \psi_{\beta}\left(1+\varpi^lx\nnpmod{\frak{p}^r}\right)
 =\tau\left(\varpi^{-l^{\prime}}T_{K/F}(x\beta)\right)
$$
with $x\in O_K=O_F[\beta]$ such that 
$x+x^{\tau}\equiv 0\npmod{\frak{p}^{l^{\prime}}}$.

The equivalence classes of the irreducible complex representations of
the finite group $G(O_F/\frak{p}^r)$ is denoted by 
$\text{\rm Irr}(G(O_F/\frak{p}^r))$ which is identified with the
equivalence classes of the representations of the compact group
$G(O_F)$ trivial on $G(\frak{p}^r)$. 

Let us denote by $\text{\rm Irr}(G(O_F/\frak{p}^r),\psi_{\beta})$ 
the equivalence classes of the irreducible complex representations 
$\delta$ of $G(O_F/\frak{p}^r)$ such that the restriction 
$\delta|_{G(\frak{p}^l/\frak{p}^r)}$ contains the character
$\psi_{\beta}$. 

The set of the group
homomorphisms $\theta:T(O_F/\frak{p}^r)\to\Bbb C^{\times}$ such that 
$\theta=\psi_{\beta}$ on 
$T(O_F/\frak{p})\cap G(\frak{p}^l/\frak{p}^r)$ is denoted by 
$\Theta(T(O_F/\frak{p}^r),\psi_{\beta})$. 

Then a bijection
$\theta\mapsto\delta_{\beta,\theta}$ of 
$\Theta(T(O_F/\frak{p}^r),\psi_{\beta})$ 
onto $\text{\rm Irr}(G(O_F/\frak{p}^r),\psi_{\beta})$ is given as an
application of the general theory developed by 
\cite{Takase2017}. See section 
\ref{sec:representation-of-hyper-compact-subgroup} for some details. 

Now our main result is

\begin{thm}\label{th:main-result}
For any $\theta\in\Theta(T(O_F/\frak{p}^r),\psi_{\beta})$, the compactly
induced representation 
$\pi_{\beta,\theta}
 =\text{\rm ind}_{G(O_F)}^{G(F)}\delta_{\beta,\theta}$ is
\begin{enumerate}
\item irreducible generic supercuspidal representation of $G(F)$, 
\item with the formal degree 
$$
 \dim\delta_{\beta,\theta}
 =q^{n^2r}(1-q^{-n})\prod_{k=1}^{n-1}\left(1-q^{-2k}\right)
$$
      if the Haar measure of $G(F)$ is normalized so that the volume of
      $G(O_F)$ is one.
\item The irreducible representation $\delta_{\beta,\theta}$ is contained in
      $\pi_{\beta,\theta}|_{G(O_F)}$ with multiplicity one, and
\item if an irreducible complex representation $\delta$ of 
      $G(O_F/\frak{p}^r)$ is contained in
      $\pi_{\beta,\theta}|_{G(O_F)}$,  
      then $\delta=\delta_{\beta,\theta}$. 
\end{enumerate}
\end{thm}

The explicit description of $\delta_{\beta,\theta}$ will be given in
the next section. The proof of Theorem \ref{th:main-result}, except
the genericity of $\pi_{\beta,\theta}$, is given in Section 
\ref{sec:construction-of-super-cuspidal-representation}. The
genericity is proved in Section
\ref{sec:genericity-of-supercuspidal-representation}. 

\section{Representations of hyperspecial compact subgroup}
\label{sec:representation-of-hyper-compact-subgroup}
In this section, we will recall the explicit descriptions of the
bijection $\theta\mapsto\delta_{\beta,\theta}$ of 
$\Theta(T(O_F/\frak{p}^r),\psi_{\beta})$ onto 
$\text{\rm Irr}(G(O_F/\frak{p}^r),\psi_{\beta})$ given by \cite{Takase2017}. 

We will keep the notations of section \ref{sec:main-results}. 

\subsection[]{}\label{subsec:clifford-theory}
Let us denote by $H(O_F/\frak{p}^r)$ the subgroup of 
$g\in G(O_F/\frak{p}^r)$ such that
$$
 \psi_{\beta}(ghg^{-1})=\psi_{\beta}(h)\,
 \text{\rm for all}\,h\in G(\frak{p}^l/\frak{p}^r).
$$
In other word
\begin{align*}
 H(O_F/\frak{p}^r)
 &=\left\{g\nnpmod{\frak{p}^r}\in G(O_F/\frak{p}^r)\biggm|
          \text{\rm Ad}(g)\beta\equiv\beta\nnpmod{\frak{p}^{l^{\prime}}}
         \right\}\\
 &=T(O_F/\frak{p}^r)\cdot G(\frak{p}^{l^{\prime}}/\frak{p}^r)
\end{align*}
because $T$ is a smooth $O_F$-group scheme, and hence the canonical
group homomorphism $T(O_F/\frak{p}^r)\to T(O_F/\frak{p}^{l^{\prime}})$
is surjective. 

Let us denote by 
$\text{\rm Irr}(H(O_F/\frak{p}^r),\psi_{\beta})$ the set of the
equivalence classes of the irreducible complex representations
$\sigma$ of $H(O_F/\frak{p}^r)$ such that the restriction 
$\sigma|_{G(\frak{p}^l/\frak{p}^r)}$ contains the character
$\psi_{\beta}$. Then Clifford's theory says
that 
$$
 \sigma\mapsto\text{\rm Ind}_{H(O_F/\frak{p}^r)}
                            ^{G(O_F/\frak{p}^r)}\sigma
$$
gives a bijection
of $\text{\rm Irr}(H(O_F/\frak{p}),\psi_{\beta})$ onto 
$\text{\rm Irr}(G(O_F/\frak{p}^r),\psi_{\beta})$. We have a bijection 
$\theta\mapsto\sigma_{\beta,\theta}$ of
$\Theta(T(O_F/\frak{p}^r),\psi_{\beta})$ onto 
$\text{\rm Irr}(H(O_F/\frak{p}^r),\psi_{\beta})$ described as in the
following two subsections. Then 
$$
 \theta\mapsto
 \delta_{\beta,\theta}=\text{\rm Ind}_H^G\sigma_{\beta,\theta}
 \quad
(G=G(O_F/\frak{p}^r), H=H(O_F/\frak{p}^r)) 
$$
gives the required
bijection of $\Theta(T(O_F/\frak{p}^r),\psi_{\beta})$ onto 
$\text{\rm Irr}(G(O_F/\frak{p}^r),\psi_{\beta})$. 

\subsection[]{}
\label{subsec:description-of-sigma-beta-theta-for-even-r}
Suppose that $r=2l$ is even. Then
$$
 H(O_F/\frak{p}^r)
 =T(O_F/\frak{p}^r)\cdot G(\frak{p}^l/\frak{p}^r)
$$
and 
$\theta\in\Theta(T(O_F/\frak{p}^r),\psi_{\beta})$ define a character 
$\sigma_{\beta,\theta}$ of $H(O_F/\frak{p}^r)$ by
$$
 \sigma_{\beta,\theta}(gh)=\theta(g)\cdot\psi_{\beta}(h)
$$
for $g\in T(O_F/\frak{p}^r)$ and 
$h\in G(\frak{p}^l/\frak{p}^r)$. 

\subsection[]{}
\label{subsec:description-of-sigma-beta-theta-for-odd-r}
Suppose that $r=2l-1\geq 3$ is odd. In this case we have
$$
 H(O_F/\frak{p}^r)
 =T(O_F/\frak{p}^r)\cdot G(\frak{p}^{l-1}/\frak{p}^r).
$$
We will construct a representation of $G(\frak{p}^{l-1}/\frak{p}^r)$
by means of Schr\"odinger representation, and then we will extend it
to a representation of $H(O_F/\frak{p}^r)$ by means of Weil
representation over the finite field $\Bbb F$. 

Let $\frak{t}=\text{\rm Lie}(T)$ be the Lie algebra of the $O_F$-group
scheme $T=Z_G(\beta)$. Then, for any commutative $O_F$-algebra $L$
$$
 \frak{t}(L)=\{X\in\frak{g}(L)\mid [X,\overline\beta]=0\}
$$
is the centralizer of $\overline\beta\in\frak{g}(L)$ (the image of
$\beta\in\frak{g}(O_F)$ by the canonical homomorphism 
$\frak{g}(O_F)\to\frak{g}(L)$) in $\frak{g}(L)$. 

Let us denote by 
$Z(\frak{p}^{l-1}/\frak{p}^r)$ the inverse image
of $\frak{t}(\Bbb F)$ under the surjective group homomorphism 
$G(\frak{p}^{l-1}/\frak{p}^r)\to\frak{g}(\Bbb F)$ defined by
$$
 1_{2n}+\varpi^{l-1}X\nnpmod{\frak{p}^r}\mapsto
 X\nnpmod{\frak{p}}
$$
which is a normal subgroup of $G(\frak{p}^{l-1}/\frak{p}^r)$
containing $G(\frak{p}^l/\frak{p}^r)$. 

Put $\Bbb V=\frak{g}(\Bbb F)/\frak{t}(\Bbb F)$ which is a symplectic
$\Bbb F$-space with respect to the symplectic form 
$\langle\dot X,\dot Y\rangle_{\beta}=B([X,Y],\overline\beta)$ 
($X,Y\in\frak{g}(\Bbb F)$ and 
 $\overline\beta=\beta\npmod{\frak p}\in\frak{g}(\Bbb F)$). 
Sice $B|_{\frak{t}(\Bbb F)\times\frak{f}(\Bbb F)}$ is non-degenerate,
we have $\frak{g}(\Bbb F)=V\oplus\frak{t}(\Bbb F)$ with
$$
 V=\{X\in\frak{g}(\Bbb F)\mid B(X,\frak{t}(\Bbb F))=0\}.
$$
Let us denote by $v\mapsto[v]$ the inverse $\Bbb F$-linear mapping of
the canonical $\Bbb F$-linear isomorphism $V\,\tilde{\to}\,\Bbb V$.

Let $H_{\beta}$ be the Heisenberg group
associated with the symplectic space $(\Bbb V,\langle\,,\rangle_{\beta})$,
that is $H_{\beta}=\Bbb V\times\Bbb C^1$ with group operation
$$
 (u,s)\cdot(v,t)
 =(u+v,s\cdot t\cdot\widehat\tau(2^{-1}\langle u,v\rangle_{\beta}))
$$
where $\widehat\tau(x\npmod{\frak p})=\tau(\varpi^{-1}x)$ and $\Bbb
C^1$ is the multiplicative group of complex numbers of absolute value
one. Fix a polarization $\Bbb V=\Bbb W^{\prime}\oplus\Bbb W$. Then we
have an irreducible unitary representation 
$(\omega_{\beta},L^2(\Bbb W^{\prime}))$ of $H_{\beta}$ called
Schr\"odinger representation. Here $L^2(\Bbb W^{\prime})$ is the
complex vector space of the complex-valued functions on $\Bbb
W^{\prime}$, and 
$$
 \left(\omega_{\beta}(u,s)f\right)(w)
 =s\cdot\widehat\tau\left(2^{-1}\langle u_-,u_+\rangle_{\beta}
                          +\langle w,u_+\rangle_{\beta}\right)
  \cdot f(w+u_-)
$$
for $f\in L^2(\Bbb W^{\prime})$ and $(u,s)\in H_{\beta}$ with
$u=u_-+u_+$ ($u_-\in\Bbb W^{\prime}, u_+\in\Bbb W$). By a general
theory of Weil representation over finite field \cite{Gerardin1977}, 
there exists a group homomorphism 
$\Omega:Sp(\Bbb V)\to GL_{\Bbb C}(L^2(\Bbb W^{\prime}))$ such that
$$
 \omega_{\beta}(u\sigma,s)
 =\Omega(\sigma)^{-1}\circ\omega_{\beta}(u,s)\circ\Omega(\sigma)
$$
for all $\sigma\in Sp(\Bbb V)$ and $(u,s)\in H_{\beta}$. 

Fix additive character $\rho$ of $\frak{t}(\Bbb F)$. Then an
irreducible representation $\omega_{\beta,\rho}$ of 
$G(\frak{p}^{l-1}/\frak{p}^r)$ on $L^2(\Bbb W^{\prime})$ 
is defined as follows. 
Take an element 
$h=1_{2n}+\varpi^{l-1}T\npmod{\frak{p}^r}$ 
of $G(\frak{p}^{l-1}/\frak{p}^r)$ and hence 
$T\npmod{\frak p}\in\frak{g}(\Bbb F)$. Put 
$T\npmod{\frak p}=[v]+Y$ with $v\in\Bbb V$ and 
$Y\in\frak{t}(\Bbb F)$. Then 
$\omega_{\beta,\rho}(h)\in GL_{\Bbb C}(L^2(\Bbb W^{\prime}))$ is
defined by 
$$
 \omega_{\beta,\rho}(h)
 =\tau\left(\varpi^{-l}B(T,\beta)-2^{-1}B(T^2,\beta)\right)\cdot
  \rho(Y)\cdot\omega_{\beta}(v,1).
$$
If $h\in Z(\frak{p}^{l-1}/\frak{p}^r)$, then 
$\omega_{\beta,\rho}(h)$ is the homothety of
$$
 \psi_{\beta,\rho}(h)
 =\tau\left(\varpi^{-l}B(T,\beta)-2^{-1}B(T^2,\beta)\right)\cdot
  \rho(T\nnpmod{\frak p}),
$$
where $\psi_{\beta,\rho}$ is a character of 
$Z(\frak{p}^{l-1}/\frak{p}^r)$ which coincides with
$\psi_{\beta}$ on $G(\frak{p}^l/\frak{p}^r)$, and all extensions of
$\psi_{\beta}$ to $Z(\frak{p}^{l-1}/\frak{p}^r)$ are
given in this way. Furthermore we have
\begin{equation}
 \text{\rm Ind}^{G(\frak{p}^{l-1}/\frak{p}^r)}
               _{Z(\frak{p}^{l-1}/\frak{p}^r)}
               \psi_{\beta,\rho}
 =\bigoplus^{\dim\omega_{\beta,\rho}}\omega_{\beta,\rho}
\label{eq:isotypic-decomposition-of-finite-induced-rep}
\end{equation}
where
$$
 \dim\omega_{\beta,\rho}
 =q^{2^{-1}(\dim_{\Bbb F}\frak{g}(\Bbb F)-\frak{t}(\Bbb F))}
 =q^{n^2}.
$$
Take a $g\in T(O_F/\frak{p}^r)$ and put
$\overline g=g \npmod{\frak p}\in T(\Bbb F)$. Then 
$\sigma_{\overline g}\in Sp(\Bbb V)$ is defined by 
$v\sigma_{\overline g}
 =\text{\rm Ad}(\overline g)^{-1}X\npmod{\frak{t}(\Bbb F)}$ for 
$v=X\npmod{\frak{t}(\Bbb F)}\in\Bbb V$. Then an irreducible representation 
$(\sigma_{\beta,\rho},L^2(\Bbb W^{\prime}))$ of 
$H(O_F/\frak{p}^r)
 =T(O_F/\frak{p}^r)\cdot G(\frak{p}^{l-1}/\frak{p}^r)$ is defined by
$$
 \sigma_{\beta,\rho}(gh)
 =\Omega(\sigma_{\overline g})\circ\omega_{\beta,\rho}(h)
$$
for $g\in T(O_F/\frak{p}^r)$ and $h\in G(\frak{p}^{l-1}/\frak{p}^r)$. 
Note that on  
$$
 T(O_F/\frak{p}^r)\cap G(\frak{p}^{l-1}/\frak{p}^r)
 \subset Z(\frak{p}^{l-1}/\frak{p}^r),
$$
$\sigma_{\beta,\rho}$ is the homothety by $\psi_{\beta,\rho}$. 

Now take a $\theta\in\Theta(T(O_F/\frak{p}^r),\psi_{\beta})$ and restrict
$\rho$ by the condition that $\psi_{\beta,\rho}=\theta$ on 
$T(O_F/\frak{p}^r)\cap G(\frak{p}^{l-1}/\frak{p}^r)$. In other word
$\rho$ is given by
$$
 \rho(Y\nnpmod{\frak p})
 =\tau\left(-\varpi^{-l}B(Y,\beta)\right)\cdot\theta(g)
$$
for $Y\in\frak{t}(O_F)$ and 
$$
 g=1_{2n}+\varpi^{l-1}Y+2^{-1}\varpi^{2l-2}Y^2\nnpmod{\frak{p}^r}
 \in T(O_F/\frak{p}^r).
$$
Note that the canonical morphism $\frak{t}(O_F)\to\frak{t}(\Bbb F)$ is
surjective. Then finally the representation 
$\sigma_{\beta,\theta}$ of $H(O_F/\frak{p}^r)$ is defined by
$$
 \sigma_{\beta,\theta}(gh)
 =\theta(g)\cdot\sigma_{\beta,\rho}(gh)
$$
for $g\in T(O_F/\frak{p}^r)$ and $h\in G(\frak{p}^{l-1}/\frak{p}^r)$. 

\subsection[]{}
\label{subsec:dimension-of-delta-beta-theta}
The explicit construction of $\beta_{\beta,\theta}$ being given, the
dimension of it is

\begin{prop}\label{prop:dimension-of-delta-beta-theta}
For any $\theta\in\Theta(T(O_F/\frak{p}^r),\psi_{\beta})$, we have
$$
  \dim\delta_{\beta,\theta}
 =q^{n^2r}(1-q^{-n})\prod_{k=1}^{n-1}\left(1-q^{-2k}\right).
$$
\end{prop}
\begin{proof}
To begin with
$$
 \dim\delta_{\beta,\theta}
 =(G(O_F/\frak{p}^r):H(O_F/\frak{p}^r))\cdot\dim\sigma_{\beta,\theta}
$$
and
$$
 \dim\sigma_{\beta,\theta}=\begin{cases}
                            1&:\text{\rm if $r=2l$ is even,}\\
                            q^{n^2}&:\text{\rm if $r=2l-1$ is odd.}
                           \end{cases}
$$
Since 
$H(O_F/\frak{p}^r)
 =T(O_F/\frak{p}^r)G(\frak{p}^{l^{\prime}}/\frak{p}^r)$, we have
\begin{align*}
 |H(O_F/\frak{p}^r)|
 &=\frac{|T(O_F/\frak{p}^r)||G(\frak{p}^{l^{\prime}})|}
        {|T(O_F/\frak{p}^r)\cap G(\frak{p}^{l^{\prime}})|}\\
 &=|T(O_F/\frak{p}^{l^{\prime}})|\cdot
   \frac{|G(O_F/\frak{p}r)|}
        {|G(O_F/\frak{p}^{l^{\prime}})}.
\end{align*}
Now
$$
 |G(O_F/\frak{p}^r)|
 =q^{n(2n+1)r}\prod_{k=1}^n(1-q^{-2k}).
$$
On the other hand we have an exact sequence
$$
 1\to T(O_F/\frak{p}^r)\to\left(O_K/\varpi^rO_K\right)^{\times}
  \xrightarrow{N_{K/K_+}}\left(O_{K_+}/\varpi^rO_{K_+}\right)^{\times}
  \to 1
$$
because of \eqref{eq:description-of-t-o-mod-p-r}, and hence 
$|T(O_F/\frak{p}^r)|=q^{nr}(1+q^{-n})$.
\end{proof}

\section{Construction of irreducible supercuspidal representations}
\label{sec:construction-of-super-cuspidal-representation}
In this section we will prove Theorem \ref{th:main-result} except the
genericity. 

We will keep the notations of the preceding sections. Fix a 
$\theta\in\Theta(T(O_F/\frak{p}^r),\psi_{\beta})$. Let us denote by
$V_{\beta,\theta}$ the representation space
of $\delta_{\beta,\theta}$. 

\subsection[]{}
\label{subsec:particular-subgroup-of-sp(2n)}
To begin with, we will fix standard parabolic subgroups of 
$G=Sp_{2n}$. 
$$
 B=\left\{\begin{bmatrix}
           a&b\\
           0&^{\tau}a^{-1}
          \end{bmatrix}\in Sp_{2n}\right\}=L\ltimes U
$$
is a Borel subgroup where 
\begin{align*}
 L&=\left\{\begin{bmatrix}
            a&0\\
            0&^{\tau}a^{-1}
           \end{bmatrix}\biggm| a=\begin{bmatrix}
                                   a_1&      &   \\
                                      &\ddots&   \\
                                      &      &a_n
                                  \end{bmatrix}\right\},\\
 U&=\left\{\begin{bmatrix}
            a&b\\
            0&^{\tau}a^{-1}
           \end{bmatrix}\in Sp_{2n}
                        \biggm|a=\begin{bmatrix}
                                  1&      &\bigastu\\
                                   &\ddots&    \\
                                   &      &1
                                 \end{bmatrix}\right\}.
\end{align*}
We use
the notation $^{\tau}a=I_n\,^taI_n$ for a square matrix $a$ of size
$n$, that is the matrix transposed with respect to the second
diagonal. The standard maximal parabolic subgroups 
$P_i=L_i\ltimes U_i$ ($1\leq i\leq n$) are given by
\begin{align*}
 L_i&=\left\{\begin{bmatrix}
              g& & & \\
               &a&b& \\
               &c&d& \\
               & & &^{\tau}g^{-1}
             \end{bmatrix}\biggm|g\in GL_i,\;
                                 \begin{bmatrix}
                                  a&b\\
                                  c&d
                                 \end{bmatrix}\in Sp_{2(n-i)}\right\},\\
 U_i&=\left\{\begin{bmatrix}
              1_i&  \ast &  \ast &\ast\\
                 &1_{n-i}&   0   &\ast\\
                 &       &1_{n-i}&\ast\\
                 &       &       &1_i
             \end{bmatrix}\in Sp_{2n}\right\}.
\end{align*}
They are all smooth $O_F$-group subscheme of $G=Sp_{2n}$. 
The Lie algebra $\frak{u}_i$ and $\frak{l}_i$ of $U_i$ and $L_i$ are
given by 
$$
 \frak{u}_i=\left\{\begin{bmatrix}
                    0&A&B&C\\
                     &0&0&^{\tau}B\\
                     & &0&-^{\tau}A\\
                     & & &  0
                   \end{bmatrix}\biggm|A,B\in M_{i,n-i}, C\in M_i, 
                                       ^{\tau}C=C\right\}
$$
and
$$
 \frak{l}_i=\left\{\begin{bmatrix}
                    X& &         &         \\
                     &A&    B    &         \\
                     &C&-^{\tau}A&         \\
                     & &         &-^{\tau}X
                   \end{bmatrix}\biggm|X\in\frak{gl}_i, 
                   \begin{bmatrix}
                    A&B\\
                    C&-^{\tau}A
                   \end{bmatrix}\in\frak{sp}_{2(n-i)}\right\}
$$
respectively. The Lie algebra of $P_i$  is 
$\frak{p}_i=\frak{l}_i\oplus\frak{u}_i$. 

\subsection[]{}
\label{subsec:remark-on-fixed-vector-of-delta-beta-theta}
We will use the following proposition repeatedly.

\begin{prop}\label{prop:remark-on-fixed-vector-of-delta-beta-theta}
For all $1\leq i\leq n$, the space of 
$U_i(\frak{p}^{r-1}/\frak{p}^r)$-fixed vectors of $V_{\beta,\theta}$
is trivial.
\end{prop}
\begin{proof}
$G(\frak{p}^l/\frak{p}^r)$ is a normal subgroup of $G(O_F/\frak{p}^r)$
  and $\delta_{\beta,\theta}|_{G(\frak{p}^l/\frak{p}^r)}$ contains the
  character $\psi_{\beta}$ of $G(\frak{p}^l/\frak{p}^r)$. Then
$$
 \delta_{\beta,\theta}|_{G(\frak{p}^l/\frak{p}^r)}
 =\left(\bigoplus_g{\psi_{\beta}}^g\right)^e
 \qquad
 (0<e\in\Bbb Z)
$$
where $\bigoplus_g$ is the direct sum over the representatives $g$ of 
$H(O_F/\frak{p}^r)\backslash G(O_F/\frak{p}^r)$ and 
$$
 {\psi_{\beta}}^g(h)=\psi_{\beta}(ghg^{-1})
 =\psi_{\text{\rm Ad}(g^{-1})\beta}(h).
$$ 
Note that 
$U_i(\frak{p}^{r-1}/\frak{p}^r)\subset G(\frak{p}^{l-1}/\frak{p}^r)$. 
If there exists a non-trivial 
$U_i(\frak{p}^{r-1}/\frak{p}^r)$-fixed vector in $V_{\beta,\beta}$,
then there exists a $g\in G(O_F)$ such that 
$\psi_{\text{\rm Ad}(g)\beta}(h)=1$ for all 
$h\in U_i(\frak{p}^{r-1}/\frak{p}^r)$. This means that 
$$
 \tau(\varpi^{-1}B(X,\text{\rm Ad}(g)\beta))=1\;
 \text{\rm that is}\;
 B(X,\text{\rm Ad}(g)\beta)\equiv 0\nnpmod{\frak p}
$$
for all $X\in\frak{u}_i(O_F)$. This implies that 
$\text{\rm Ad}(g)\beta\npmod{\frak p}\in\frak{p}_i(\Bbb F)$, and this
means that the characteristic polynomial 
$\chi_{\beta}(t)\npmod{\frak p}$ is reducible in $\Bbb F[t]$,
contradicting the assumption on $\beta$. 
\end{proof}

\subsection[]{}
\label{subsec:admissibility-supercuspidality-multiplicity-one-of-delta}
In this subsection, we will prove the statements 1), 3) and 4) of
Theorem \ref{th:main-result} except the genericity.

To begin with

\begin{prop}\label{prop:admissibility}
$E=\text{\rm ind}_{G(O_F)}^{G(F)}\delta_{\beta,\theta}$ is admissible 
$G(F)$-module.
\end{prop}
\begin{proof}
For any $0<m\in\Bbb Z$, the space of the $G(\frak{p}^m)$-fixed vectors
is
$$
 E^{G(\frak{p}^m)}
 =\bigoplus_g{V_{\beta,\theta}}^{G(O_F)\cap g^{-1}G(\frak{p}^m)g}
$$
where $\bigoplus_g$ is the direct sum over the representatives $g$ of 
the double cosets $G(\frak{p}^m)\backslash G(F)/G(\frak{p}^m)$. Pick
up one such representative $g$ which should be of the form 
$g=k\begin{bmatrix}
     a&0\\
     0&^{\tau}a^{-1}
    \end{bmatrix}$ with $k\in G(O_F)$ and
$$
 a=\begin{bmatrix}
    \varpi^{e_1}&      &            \\
                &\ddots&            \\
                &      &\varpi^{e_n}
   \end{bmatrix},
 \quad
 e_1\geq\cdots\geq e_n\geq 0.
$$
Assume that 
$$
 \text{\rm Max}\{e_k-e_{k+1}\mid 1\leq k<n\}=e_i-e_{i+1}\geq m.
$$
Then we have
$$
 \begin{bmatrix}
  a&0\\
  0&^{\tau}a^{-1}
 \end{bmatrix}U_i(O_F)\begin{bmatrix}
                       a&0\\
                       0&^{\tau}a^{-1}
                      \end{bmatrix}^{-1}\subset G(\frak{p}^m)
$$
and hence $U_i(O_F)\subset G(O_F)\cap g^{-1}G(\frak{p}^m)g$. Then
$$
 {V_{\beta,\theta}}^{G(O_F)\cap g^{-1}G(\frak{p}^m)g}
 \subset {V_{\beta,\theta}}^{U_i(O_F)}=\{0\}
$$
by Proposition
\ref{prop:remark-on-fixed-vector-of-delta-beta-theta}. Now the number
of the elements $k\begin{bmatrix}
                   a&0\\
                   0&^{\tau}a^{-1}
                  \end{bmatrix}$ where $k$ is the representative of 
$G(\frak{p}^m)\backslash G(O_F)$ and 
$$
 a=\begin{bmatrix}
    \varpi^{e_1}&      &            \\
                &\ddots&            \\
                &      &\varpi^{e_n}
   \end{bmatrix}
 \;\text{\rm with}\;
 \begin{cases}
  e_1\geq\cdots\geq e_n\geq 0,\\
  \text{\rm Max}\{e_k-e_{k+1}\mid 1\leq k<n\}\leq m
 \end{cases}
$$
is finite. So 
$E^{G(\frak{p}^m)}$ is finite dimensional.
\end{proof}

Next we will prove

\begin{prop}\label{prop:delta-is-mutiplicity-one}
As a $G(O_F)$-module 
\begin{enumerate}
\item $\text{\rm ind}_{G(O_F)}^{G(F)}\delta_{\beta,\theta}$
      contains $\delta_{\beta,\theta}$ with multiplicity one, 
\item if an irreducible representation $\delta$ of $G(O_F/\frak{p}^r)$
      is contained in 
      $\text{\rm ind}_{G(O_F)}^{G(F)}\delta_{\beta,\theta}$,
      then $\delta=\delta_{\beta,\theta}$.
\end{enumerate}
\end{prop}
\begin{proof}
Take an irreducible representation $\delta$ of 
$G(O_F/\frak{p}^r)$ which is identified with a representation of
$G(O_F)$ via the canonical surjection $G(O_F)\to G(O_F/\frak{p}^r)$. 
We have
$$
 \left.
 \left(\text{\rm ind}_{G(O_F)}^{G(F)}\delta_{\beta,\theta}\right)
 \right|_{G(O_F)}
 =\bigoplus_g\text{\rm ind}_{G(O_F)\cap g^{-1}G(O_F)g}^{G(O_F)}
   {\delta_{\beta,\theta}}^g,
$$
where $\bigoplus_g$ is the direct sum over the representatives $g$ of
the double cosets $G(O_F)\backslash G(F)/G(O_F)$ and
${\delta_{\beta,\theta}}^g(h)=\delta_{\beta,\theta}(ghg^{-1})$. Then 
we have
\begin{align*}
 \text{\rm Hom}_{G(O_F)}\left(\delta,
        \text{\rm ind}_{G(O_F)}^{G(F)}\delta_{\beta,\theta}\right)
 &=\bigoplus_g\text{\rm Hom}_{G(O_F)}\left(\delta,
     \text{\rm ind}_{G(O_F)\cap g^{-1}G(O_F)g}^{G(O_F)}
       {\delta_{\beta,\theta}}^g\right)\\
 &=\bigoplus_g\text{\rm Hom}_{G(O_F)\cap g^{-1}G(O_F)g}\left(\delta,
       {\delta_{\beta,\theta}}^g\right)\\
 &=\bigoplus_g\text{\rm Hom}_{g^{-1}G(O_F)g\cap G(O_F)}
    \left(\delta^g,\delta_{\beta,\theta}\right)
\end{align*}
by Frobenius reciprocity. Take a representative $g$ of 
$G(O_F)\backslash G(F)/G(O_F)$ such that $g\not\in G(O_F)$. Then
we can assume that 
$$
 g=\begin{bmatrix}
    a&0\\
    0&^{\tau}a^{-1}
   \end{bmatrix}\;\text{\rm with}\;
 a=\begin{bmatrix}
    \varpi^{e_1}&      &            & &      & \\
                &\ddots&            & &      & \\
                &      &\varpi^{e_i}& &      & \\
                &      &            &1&      & \\
                &      &            & &\ddots& \\
                &      &            & &      &1
   \end{bmatrix}
$$
and $e_1\geq\cdots\geq e_i>0$ ($1\leq i\leq n$). In this case
$$
 g^{-1}G(O_F)g\cap G(O_F)\supset
 g^{-1}U_i(O_F)g\cap U_i(O_F)=U_i(O_F)\supset U_i(\frak{p}^{r-1})
$$
so that
$$
 \text{\rm Hom}_{g^{-1}G(O_F)g\cap G(O_F)}(\delta^g,
   \delta_{\beta,\theta})
 \subset\text{\rm Hom}_{U_i(\frak{p}^{r-1})}(\delta^g,
   \delta_{\beta,\theta}).
$$
Now we have $gU_i(\frak{p}^{r-1})g^{-1}\subset U_i(\frak{p}^r)$ and
$\delta$ factors through $G(O_F)\to G(O_F/\frak{p}^r)$, hence
$$
 \text{\rm Hom}_{g^{-1}G(O_F)g\cap G(O_F)}(\delta^g,
   \delta_{\beta,\theta})
 =\{0\}
$$
because the space of $U_i(\frak{p}^{r-1})$-invariant vectors in
$V_{\beta,\theta}$ is trivial. Then we have
$$
 \text{\rm Hom}_{G(O_F)}(\delta,
   \text{\rm ind}_{G(O_F)}^{G(F)}\delta_{\beta,\theta})
 =\text{\rm Hom}_{G(O_F)}(\delta,\delta_{\beta,\theta})
$$
and this complete the proof.
\end{proof}

Then we have

\begin{prop}\label{prop:irreducibility-of-ind-delta}
$\text{\rm ind}_{G(O_F)}^{G(F)}\delta_{\beta,\theta}$ is an irreducible
  representation of $G(F)$.
\end{prop}
\begin{proof}
Put $\delta=\delta_{\beta,\theta}$ for simplicity. 
Suppose that $E=\text{\rm ind}_{G(O_F)}^{G(F)}\delta$
has a non-trivial $G(F)$-submodule $V$. Then we have 
$$
 \text{\rm Hom}_{G(F)}
 \left(V,\text{\rm ind}_{G(O_F)}^{G(F)}\delta\right)
 \subset
 \text{\rm Hom}_{G(F)}
 \left(V,\text{\rm Ind}_{G(O_F)}^{G(F)}\delta\right)
 =\text{\rm Hom}_{G(O_F)}(V,\delta)
$$
by Frobenius reciprocity 
so that $V$ contains $\delta$ as a $G(O_F)$-module. On
the other hand $M=E/V$ is an admissible $G(F)$-module and we have
\begin{align*}
 \text{\rm Hom}_{G(F)}\left(\text{\rm ind}_{G(O_F)}^{G(F)}\delta,M\right)
 &=\text{\rm Hom}_{G(F)}\left(M^{\vee},
    \left(\text{\rm ind}_{G(O_F)}^{G(F)}\delta\right)^{\vee}\right)\\
 &=\text{\rm Hom}_{G(F)}\left(M^{\vee},
    \text{\rm Ind}_{G(O_F)}^{G(F)}\delta^{\vee}\right)\\
 &=\text{\rm Hom}_{G(O_F)}\left(M^{\vee},\delta^{\vee}\right)\\
 &=\text{\rm Hom}_{G(O_F)}\left(\delta^{\vee\vee},
    \left(M^{\vee}|_{G(O_F)}\right)^{\vee}\right)\\
 &=\text{\rm Hom}_{G(O_F)}(\delta,M)
\end{align*}
(here $\vee$ denote the contragredient representation) 
so that $M$ contains $\delta$ as a $G(O_F)$-module. Since $E$ is
semisimple $G(O_F)$-module, this means that $E$ contains $\delta$ at
least twice, contradicting to Proposition 
\ref{prop:delta-is-mutiplicity-one}.
\end{proof}

Now the compactly induced representation 
$\pi_{\beta,\theta}
 =\text{\rm ind}_{G(O_F)}^{G(F)}\delta_{\beta,\theta}$ is an
 irreducible admissible representation of $G(F)$, it is well known
 that $\pi_{\beta,\theta}$ is supercuspidal. 

\subsection[]{}\label{subsec:formal-degree}
In this subsection, we will prove the statement 2) of Theorem
\ref{th:main-result}. 
Put $(\delta_{\beta,\theta},V_{\beta,\theta})=(\delta,V)$ for
simplicity. 

Let $d_{G(O_F)}(k)$ be the Haar measure on $G(O_F)$ such that 
the volume of $G(O_F)$ is one, and $d_{G(F)}(x)$ the Haar measure on
$G(F)$ such that 
$$
 \int_{G(F)}\varphi(x)d_{G(F)}(x)
 =\sum_{\dot x\in G(F)/G(O_F)}\int_{G(O_F)}\varphi(xk)d_{G(O_F)}(k)
$$
for all compactly supported complex valued continuous functions
$\varphi$. Then the volume of $G(O_F)$ with respect to $d_{G(F)}(x)$ is
one. 

Now $\text{\rm ind}_{G(O_F)}^{G(F)}\delta$ consists of
the smooth function $\varphi:G(F)\to V$ such that
\begin{enumerate}
\item $\varphi(xk)=\delta(k)^{-1}\varphi(x)$ for all $k\in G(O_F)$ and
\item $\text{\rm supp}(\varphi)\pmod{G(O_F)}$ is compact in 
      $G(F)/G(O_F)$.
\end{enumerate}
On the other hand $\text{\rm Ind}_{G(O_F)}^{G(F)}\delta^{\vee}$
consists of the smooth functions
$$
 \psi:G(F)\to V^{\ast}=\text{\rm Hom}_{\Bbb C}(V,\Bbb C)
$$
such that $\psi(xk)=\psi(x)\circ\delta(k)$ for all $k\in G(O_F)$. A
non-degenerate pairing
$$
 \langle\,,\rangle:
 \text{\rm ind}_{G(O_F)}^{G(F)}\delta\times
 \text{\rm Ind}_{G(O_F)}^{G(F)}\delta^{\vee}\to\Bbb C
$$
is defined by
$$
 \langle\varphi,\psi\rangle
 =\sum_{\dot x\in G(F)/G(O_F)}\langle\varphi(x),\psi(x)\rangle
$$
where $\langle\,,\rangle:V\times V^{\ast}\to\Bbb C$ is the canonical
pairing. 

Take any $v\in V$ and $\alpha\in V^{\ast}$ and define 
$\varphi_v\in\text{\rm ind}_{G(O_F)}^{G(F)}\delta$ and 
$\psi_{\alpha}\in\text{\rm Ind}_{G(O_F)}^{G(F)}\delta^{\vee}$
by
$$
 \varphi_v(x)=\begin{cases}
               \delta(x)^{-1}v&:\text{\rm if $x\in G(O_F)$,}\\
               0              &:\text{\rm otherwise},
              \end{cases}
$$
and by
$$
 \psi_{\alpha}(x)
  =\begin{cases}
    \alpha\circ\delta(x)&:\text{\rm if $x\in G(O_F)$,}\\
    0                   &:\text{\rm otherwise}
   \end{cases}
$$
respectively. Then we have
$$
 \langle g\cdot\varphi_v,\psi_{\alpha}\rangle
 =\begin{cases}
   \langle\delta(g)v,\alpha\rangle&:\text{\rm if $g\in G(O_F)$,}\\
   0                              &:\text{\rm otherwise}
  \end{cases}
$$
for any $g\in G(F)$. Now
\begin{align*}
 &\int_{G(F)}\langle g\cdot\varphi_v,\psi_{\alpha}\rangle
             \langle g^{-1}\cdot\varphi_v,\psi_{\alpha}\rangle
               d_{G(F)}(g)\\
 =&\sum_{\dot g\in G(F)/G(O_F)}\int_{G(O_F)}
    \langle gk\cdot\varphi_v,\psi_{\alpha}\rangle
    \langle k^{-1}g^{-1}\cdot\varphi_v,\psi_{\alpha}\rangle
               d_{G(O_F)}(k)\\
 =&\sum_{\dot g\in G(F)/G(O_F)}\int_{G(O_F)}
    \langle g\cdot\varphi_{\delta(k)v},\psi_{\alpha}\rangle
    \langle g^{-1}\cdot\varphi_v,\psi_{\delta^{\vee}(k)\alpha}\rangle
               d_{G(O_F)}(k)\\
 =&\int_{G(O_F)}\langle\delta(k)v,\alpha\rangle
                \langle v,\delta^{\vee}(k)\alpha\rangle
                    d_{G(O_F)}(k)\\
 =&(\dim\delta)^{-1}\langle v,\alpha\rangle^2
 =(\dim\delta)^{-1}\langle \varphi_v,\psi_{\alpha}\rangle^2
\end{align*}
so that the formal degree of 
$\text{\rm ind}_{G(O_F)}^{G(F)}\delta$ is 
$$
 \dim\delta
 =q^{n^2r}(1-q^{-n})\prod_{k=1}^{n-1}(1-q^{-2k})
$$
by Proposition \ref{prop:dimension-of-delta-beta-theta}.

\section{Genericity of supercuspidal representations}
\label{sec:genericity-of-supercuspidal-representation}
In this section we will show that the irreducible supercuspidal
representation 
$\pi_{\beta,\theta}
 =\text{\rm ind}_{G(O_F)}^{G(F)}\delta_{\beta,\theta}$ constructed in
 the preceding section is generic, that is 
$$
 \text{\rm Hom}_{G(F)}\left(\pi_{\beta,\theta},
   \text{\rm Ind}_{U(F)}^{G(F)}\chi\right)\neq 0
$$
for some generic character $\chi$ of $U(F)$.

\subsection[]{}\label{subsec:generic-character-of-u(f)}
The characters (that is the continuous group homomorphism to 
$\Bbb C^1$) of $U(F)$ are given by
$$
 \chi_u(h)
 =\tau\left(\sum_{k=1}^{n-1}u_ka_{k,k+1}+u_nb_{n,1}\right)
 \;\text{\rm for}\;
 h=\begin{bmatrix}
    a&b\\
    0&^{\tau}a^{-1}
   \end{bmatrix}\in U(F)
$$
with $u=(u_1,\cdots,u_{n-1},u_n)\in F^n$ where 
$a_{ij}$ and $b_{ij}$ denote the $(i,j)$-entry of the square matrices
$a$ and $b$ respectively. Put $\chi_u^g(h)=\chi_u(g^{-1}hg)$ 
for any $g\in G(F)$. If $g=\begin{bmatrix}
                            t&0\\
                            0&^{\tau}t^{-1}
                           \end{bmatrix}\in T(F)$, then 
$\chi_u^g=\chi_{u^{\prime}}$ with
$$
 u^{\prime}=(t_1^{-1}t_2u_1,\cdots,t_{n-1}^{-1}t_nu_{n-1},t_n^{-2}u_n).
$$
So $\chi_u$ is generic, that is
$$
 \{g\in T(F)\mid\chi_u^g=\chi_u\}=Z(G(F))=\{\pm 1_{2n}\}
$$
if and only if $u_i\neq 0$ ($1\leq i\leq n$). In this case
$$
 \chi_u^g=\chi_{(2,2,\cdots,2,\lambda)}
 \quad
 (\lambda\in F^{\times})
$$
with suitable $g\in T(F)$.

\subsection[]{}\label{subsec:first-step-for-genericity}
Put $u=(2,2,\cdots,2,\varpi^e)$. Because of Iwasawa decomposition 
$G(F)=G(O_F)T(F)U(F)$, we have
\begin{align*}
 \text{\rm Hom}_{G(F)}&\left(
  \text{\rm ind}_{G(O_F)}^{G(F)}\delta_{\beta,\theta},
  \text{\rm Ind}_{U(F)}^{G(F)}\chi_u\right)\\
 &=\prod_g\text{\rm Hom}_{G(O_F)\cap gU(F)g^{-1}}
                            (\delta_{\beta,\theta},\chi_u^g)\\
 &=\prod_g\text{\rm Hom}_{U(O_F)}(\delta_{\beta,\theta},\chi_u^g)
\end{align*}
where the last $\prod_g$ is the direct product over the
representatives $g\in T(F)$ of $G(O_F)\backslash G(F)/U(F)$. We can
put $g=\begin{bmatrix}
        \varpi^m&0\\
        0&^{\tau}\varpi^{-m}
       \end{bmatrix}$ with 
$$
 \varpi^m=\begin{bmatrix}
           \varpi^{m_1}&      &            \\
                       &\ddots&            \\
                       &      &\varpi^{m_n}
          \end{bmatrix},
 \qquad
 m=(m_1,\cdots,m_n)\in\Bbb Z^n.
$$
Then we have 

\begin{prop}\label{prop:only-if-chi-u-g-is-chi-f}
$\text{\rm Hom}_{U(O_F)}(\delta_{\beta,\theta},\chi_u^g)\neq 0$ for
  some $g$ only if $e\equiv r\npmod{2}$. In this case 
$$
 \chi_u^g
 =\chi_{(2\varpi^{-r},2\varpi^{-r},\cdots,2\varpi^{-r},\varpi^{-r})}.
$$
\end{prop}
\begin{proof}
Put  where 
Then $\chi_u^g=\chi_{u^{\prime}}$ with
$$
 u^{\prime}
 =(2\varpi^{m_2-m_1},\cdots,2\varpi^{m_n-m_{n-1}},\varpi^{e-2m_n}).
$$
If $m_{i+1}-m_i\geq -(r-1)$ for some $1\leq i<n$ or 
$e-2m_n\geq -(r-1)$, then $\chi_u^g(h)=1$ for all 
$h\in U_i(\frak{p}^{r-1})$ for some $1\leq i\leq n$. In other word 
$\delta_{\beta,\theta}$ contains non-trivial
$U_i(\frak{p}^{r-1})$-invariant vectors which contradicts to 
Proposition \ref{prop:remark-on-fixed-vector-of-delta-beta-theta}. So
we have
$$
 m_{i+1}-m_i\leq -r
 \quad
 (1\leq i<n)
 \;\text{\rm and}\;
 e-2m_n\leq -r.
$$
On the other hand $\delta_{\beta,\theta}$ is trivial on
$G(\frak{p}^r)$, and hence $\chi_u^g(h)=1$ for all 
$h\in U(O_F)\cap G(\frak{p}^r)$. This means that
$$
 m_{i+1}-m_i+r\geq 0
 \quad
 (1\leq i<n)
 \;\text{\rm and}\;
 e-2m_n+r\geq 0.
$$
Hence $e+r=2m_n$ and $m_i-m_{i+1}=r$ ($1\leq i<n$). 
\end{proof}

Put
$\chi_r=\chi_{(2\varpi^{-r},\cdots,2\varpi^{-r},\varpi^{-r})}$ which
is regarded as a character of $U(O_F/\frak{p}^r)$. On the other hand
we have
$$
 \delta_{\beta,\theta}
 =\text{\rm Ind}_{H(O_F/\frak{p}^r)}^{G(O_F/\frak{p}^r)}
   \sigma_{\beta,\theta}
$$
which is considered as a representation of $G(O_F)$ via the canonical
surjection $G(O_F)\to G(O_F/\frak{p}^r)$. Then, putting 
$U=U(O_F/\frak{p}^r)$ and $H=H(O_F/\frak{p}^r)$ for the sake of
simplicity, we have
\begin{align*}
 \text{\rm Hom}_{U(O_F)}\left(\delta_{\beta,\theta},\chi_r\right)
 &=\text{\rm Hom}_{U}\left(
    \text{\rm Ind}_{H}^{G(O_F/\frak{p}^r)}
      \sigma_{\beta,\theta},\chi_r\right)\\
 &=\bigoplus_g
    \text{\rm Hom}_{U}\left(
     \text{\rm Ind}_{U\cap gHg^{-1}}
                   ^{U}
      \sigma_{\text{\rm Ad}(g)\beta,\theta^g},\chi_r\right)\\
 &=\bigoplus_g
    \text{\rm Hom}_{U\cap gHg^{-1}}
     \left(\sigma_{\text{\rm Ad}(g)\beta,\theta^g},\chi_r\right)
\end{align*}
where $\bigoplus_g$ is the direct sum over the representatives of 
$U\backslash G(O_F/\frak{p}^r)/H$ and 
$\theta^g(h)=\theta(g^{-1}hg)$ ($h\in T(O_F/\frak{p}^r)$). Note that
$$
 U(O_F/\frak{p}^r)\cap gH(O_F/\frak{p}^r)g^{-1}
 =U(\frak{p}^{l^{\prime}}/\frak{p}^r)
 \;\text{\rm for all $g\in G(O_F/\frak{p}^r)$}.
$$
To show this fact, 
replacing $\text{\rm Ad}(g)\beta$ with $\beta$, it is sufficient to
consider the case $g=1_{2n}$. In this case 
$U(O_F/\frak{p}^r)\cap H(O_F/\frak{p}^r)$ is the inverse image 
of $U(O_F/\frak{p}^{l^{\prime}})\cap T(O_F/\frak{p}^{l^{\prime}})$
under the canonical surjection 
$U(O_F/\frak{p}^r)\to U(O_F/\frak{p}^{l^{\prime}})$. Take a
$$
 h=1_{2n}+X\nnpmod{\frak{p}^{l^{\prime}}}
 \in U(O_F/\frak{p}^{l^{\prime}})\cap T(O_F/\frak{p}^{l^{\prime}})
 =U(O_F/\frak{p}^{l^{\prime}})\cap
   \left(O_K/\varpi^{l^{\prime}}O_K\right)^{\times}.
$$
Then 
$X\npmod{\frak{p}^{l^{\prime}}}\in O_K/\varpi^{l^{\prime}}O_K$ is
nilpotent, and hence $X\equiv 0\npmod{\frak{p}^{l^{\prime}}}$. So 
$U(O_F/\frak{p}^{l^{\prime}})\cap T(O_F/\frak{p}^{l^{\prime}})
 =\{1\}$ and we have
$$
 U(O_F/\frak{p}^r)\cap H(O_F/\frak{p}^r)
 =U(\frak{p}^{l^{\prime}}/\frak{p}^r).
$$

\subsection[]{}
\label{subsec:condition-psi-beta-equal-to-chi-f-for-even-r}
Suppose that $r=2l$ is even, and hence $l^{\prime}=l$. 
In this case the character $\sigma_{\beta,\theta}$ of 
$H(O_F/\frak{p}^r)=T(O_F/\frak{p}^r)G(\frak{p}^l/\frak{p}^r)$ is
defined by 
$$
 \sigma_{\beta,\theta}(gh)=\theta(g)\psi_{\beta}(h)
 \;\text{\rm for}\;
 g\in T(O_F/\frak{p}^r),\;\;h\in G(\frak{p}^l/\frak{p}^r).
$$
Then we have
$$
 \text{\rm Hom}_{U(\frak{p}^l/\frak{p}^r)}
  (\sigma_{\beta,\theta},\chi_r)
 =\text{\rm Hom}_{U(\frak{p}^l/\frak{p}^r)}(\psi_{\beta},\chi_r).
$$
Hence 
$\text{\rm Hom}_{U(\frak{p}^l/\frak{p}^r)}
 (\sigma_{\beta,\theta},\chi_r)\neq 0$ if and only if 
$\psi_{\beta}=\chi_r$ on $U(\frak{p}^l/\frak{p}^r)$. 

\subsection[]{}
\label{subsec:condition-psi-beta-equal-to-chi-f-for-odd-r}
Suppose that $r=2l-1$ is odd, and hence $l^{\prime}=l-1$. In this case
the irreducible representation $\sigma_{\beta,\theta}$ of 
$H(O_F/\frak{p}^r)=T(O_F/\frak{p}^r)G(\frak{p}^{l-1}/\frak{p}^r)$ is
defined by
$$
 \sigma_{\beta,\theta}(gh)
 =\theta(g)\cdot
   \Omega(\sigma_{\overline g})\circ\omega_{\beta,\rho}(h)
$$
for $g\in T(O_F/\frak{p}^r)$ and $h\in G(\frak{p}^{l-1}/\frak{p}^r)$
with the notations of subsection 
\ref{subsec:description-of-sigma-beta-theta-for-odd-r}. Hence 
$\sigma_{\beta,\theta}=\omega_{\beta,\rho}$ on 
$U(\frak{p}^{l-1}/\frak{p}^r)$. On the other hand we have
\eqref{eq:isotypic-decomposition-of-finite-induced-rep}, hence 
$$
 \text{\rm Hom}_{U(\frak{p}^{l-1}/\frak{p}^r)}
  (\sigma_{\beta,\theta},\chi_r)\neq 0
$$
if and only if
\begin{equation}
 \text{\rm Hom}_{U(\frak{p}^{l-1}/\frak{p}^r)}
  (\text{\rm Ind}_{Z(\frak{p}^{l-1}/\frak{p}^r)}
                 ^{G(\frak{p}^{l-1}/\frak{p}^r)}
                   \psi_{\beta,\rho},\chi_r)\neq 0.
\label{eq:induction-from-z-intertwines-with-chi-f}
\end{equation}
Note that 
$$
 U(\frak{p}^{l-1}/\frak{p}^r)\cap Z(\frak{p}^{l-1}/\frak{p}^r)
 =U(\frak{p}^l/\frak{p}^r)
$$ 
because 
$1_{2n}+\varpi^{l-1}X\npmod{\frak{p}^r}\in 
 U(\frak{p}^{l-1}/\frak{p}^r)\cap Z(\frak{p}^{l-1}/\frak{p}^r)$
 implies that $X\npmod{\frak p}\in\frak{t}(\Bbb F)$ is nilpotent,
 that is $X\npmod{\frak p}=0$. Hence we have
\begin{align*}
 \text{\rm Hom}_{U(\frak{p}^{l-1}/\frak{p}^r)}&
   \left(\text{\rm Ind}_{Z(\frak{p}^{l-1}/\frak{p}^r)}
                  ^{G(\frak{p}^{l-1}/\frak{p}^r)}
                    \psi_{\beta,\rho},\chi_r\right)\\
 =&\bigoplus_g\text{\rm Hom}_{U}
    \left(\text{\rm Ind}_{U \cap gZg^{-1}}^{U}
      \psi_{\text{\rm Ad}(g)\beta,\rho\circ\text{\rm Ad}(g^{-1})},
            \chi_r\right)\\
 =&\bigoplus_g\text{\rm Hom}_{U(\frak{p}^l/\frak{p}^r)}
    \left(\psi_{\text{\rm Ad}(g)\beta,\rho\circ\text{\rm Ad}(g^{-1})},
       \chi_r\right)\\
 =&\bigoplus_g\text{\rm Hom}_{U(\frak{p}^l/\frak{p}^r)}
     \left(\psi_{\text{\rm Ad}(g)\beta},\chi_r\right)
\end{align*}
where $U=U(\frak{p}^{l-1}/\frak{p}^r)$, 
$Z=Z(\frak{p}^{l-1}/\frak{p}^r)$ and 
$\bigoplus_g$ is the direct sum over the representatives $g$ of
the double cosets $U\backslash G(\frak{p}^{l-1}/\frak{p}^r)/Z$. 
So \eqref{eq:induction-from-z-intertwines-with-chi-f} is equivalent to 
$$
 \psi_{\text{\rm Ad}(g)\beta}=\chi_r
 \;\text{\rm on $U(\frak{p}^l/\frak{p}^r)$ for some 
         $g\in G(\frak{p}^{l-1}/\frak{p}^r)$.}
$$

\subsection[]{}
\label{subsec:completion-of-proof-of-genericity}
The combination of the results of subsections 
\ref{subsec:first-step-for-genericity}, 
\ref{subsec:condition-psi-beta-equal-to-chi-f-for-even-r} and
\ref{subsec:condition-psi-beta-equal-to-chi-f-for-odd-r} implies that 
$$
 \text{\rm Hom}_{G(F)}(\pi_{\beta,\theta},
   \text{\rm Ind}_{U(F)}^{G(F)}\chi_u)\neq 0
$$
with $u=(2,2,\cdots,2,\varpi^e)$ ($e\equiv r\npmod{2}$) 
if and only if 
$$
 \psi_{\text{\rm Ad}(g)\beta}
 =\chi_{(2\varpi^{-r},\cdots,2\varpi^{-r},\varpi^{-r})}
 \;\text{\rm on $U(\frak{p}^l/\frak{p}^r)$}
$$
for some $g\in G(O_F)$. We will show that this is the case by proving
the following proposition.

\begin{prop}\label{prop:canonical-form-over-g(o-f)}
There exists a $g\in G(O_F)$ such that 
$\text{\rm Ad}(g)\beta=\begin{bmatrix}
                        A&\ast\\
                        C&-^{\tau}A
                       \end{bmatrix}$ with
$$
 A=\begin{bmatrix}
    \ast&      &      &\bigastu\\
     1  & \ast &      &    \\
        &\ddots&\ddots&    \\
    \bigzerol&      &  1   &\ast
   \end{bmatrix},
 \qquad
 C=\begin{bmatrix}
    0     &\cdots&0&1\\
    0     &\cdots&0&0\\
    \vdots&      &\vdots&\vdots\\
    0     &\cdots&0&0
   \end{bmatrix}.
$$
\end{prop}
\begin{proof}
$K$ is the splitting field of the characteristic polynomial 
$\chi_{\beta}(t)$ of
  $\beta\in\frak{g}(O_F)\subset\frak{gl}_{2n}(O_F)$, and we can put
$$
 \chi_{\beta}(t)=\prod_{i=1}^n(t^2-\lambda_i^2)
 \qquad
 (0\neq \lambda_i\in K).
$$
Let $\langle u,v\rangle=uJ_n^tv$ be the symplectic form on $K^{2n}$
and $V_{\lambda}\subset K^{2n}$ the eigen space of $\beta$ with eigen
value $\lambda\in K$. Then 
$\langle V_{\lambda},V_{\mu}\rangle\neq 0$ only if $\lambda+\mu=0$,
and we have
$$
 K^{2n}=V_{\lambda_1}\oplus\cdots\oplus V_{\lambda_n}\oplus
        V_{-\lambda_n}\oplus\cdots\oplus V_{-\lambda_1}.
$$
This means that there exists a symplectic $K$-basis 
$\{u_1,\cdots,u_n,v_n,\cdots,v_1\}$ of $K^{2n}$ such that the
representation matrix of $\beta$ with respect to the basis is 
$$
 \begin{bmatrix}
  \Lambda&0\\
  0&-^{\tau}\Lambda
 \end{bmatrix}
 \;\text{\rm with}\;
 \Lambda=\begin{bmatrix}
          \lambda_1&      &         \\
                   &\ddots&         \\
                   &      &\lambda_n
         \end{bmatrix}.
$$
In other word, there exists a $g\in G(K)$ such that 
$g\beta g^{-1}=\begin{bmatrix}
                \Lambda&0\\
                0&-^{\tau}\Lambda
               \end{bmatrix}$.

We can choose 
$$
 \beta^{\prime}
 =\begin{bmatrix}
   A&0\\
   C&-^{\tau}A
  \end{bmatrix}+\beta_0\in\frak{g}(O_F)
$$
such that $\chi_{\beta^{\prime}}(t)=\chi_{\beta}(t)$ with 
$$
 A=\begin{bmatrix}
        0    &      &      &\bigzerou\\
        1    &   0  &      &         \\
             &\ddots&\ddots&         \\
    \bigzerol&      &   1  &    0
   \end{bmatrix},
 \qquad
 C=\begin{bmatrix}
    0&\cdots&0&1\\
    0&\cdots&0&0\\
    \vdots& &\vdots&\vdots\\
    0&\cdots&0&0
   \end{bmatrix},
$$
and the entries of $\beta_0$ are zero except the elements of the first
row and of the last column. Then there exists a $g\in G(K)$ such that 
$g\beta^{\prime} g^{-1}=\beta$. Because $T=Z_G(\beta)$ splits over
$K$, the first Galois cohomology group $H^1(\text{\rm Gal}(K/F),T(K))$
is trivial. So there exists a $g\in G(F)$ such that 
$g\beta^{\prime} g^{-1}=\beta$. Due to Iwasawa decomposition 
$G(F)=G(O_F)T(F)U(F)$, we can put $g=ktu$ with $k\in G(O_F)$, 
$u\in U(F)$ and 
$$
 t=\begin{bmatrix}
    \varpi^{-m}&0\\
    0&^{\tau}\varpi^{-m}
   \end{bmatrix}\in T(F)
 \;\text{\rm with $m=(m_1.\cdots,m_n)\in\Bbb Z^n$.}
$$
Then $tu\beta^{\prime}u^{-1}t^{-1}=k^{-1}\beta k\in\frak{g}(O_F)$ and
$$
 tu\beta^{\prime}u^{-1}t^{-1}=\begin{bmatrix}
                               X&\ast\\
                               Y&-^{\tau}X
                              \end{bmatrix}
$$
with
$$
 X=\begin{bmatrix}
        \ast        & \ast &\cdots              &\ast\\
    \varpi^{m_2-m_1}& \ast &\cdots              &\ast\\
                    &\ddots&\ddots              &\vdots\\\
                    &      &\varpi^{m_n-m_{n-1}}&\ast
   \end{bmatrix},
 \qquad
 Y=\begin{bmatrix}
    0&\cdots&0&\varpi^{-2m_n}\\
    0&\cdots&0&0\\
    \vdots& &\vdots&\vdots\\
    0&\cdots&0&0
   \end{bmatrix}.
$$
This implies that $m_1\leq m_2\leq\cdots\leq m_n\leq 0$. On the other hand 
$$
 \chi_{\beta^{\prime}}(t)\nnpmod{\frak p}
 =\chi_{\beta}(t)\nnpmod{\frak p}\in\Bbb F[t]
$$
is irreducible, we have $m_1=m_2=\cdots=m_n=0$. 
\end{proof}

\subsection[]{}\label{subsec:final-remark-on-genericity}
We have proved the followings;
\begin{enumerate}
\item the irreducible supercuspidal representations 
      $\pi_{\beta,\theta}
       =\text{\rm ind}_{G(O_F)}^{G(F)}\delta_{\beta,\theta}$ are
      generic,
\item $\text{\rm Hom}_{G(F)}\left(\pi_{\beta,\theta},
           \text{\rm Ind}_{U(F)}^{G(F)}\chi_u\right)\neq 0$ with 
      $u=(2,2,\cdots,2,\varpi^e)$ if and only if 
      $e\equiv r\npmod{2}$.
\end{enumerate}


Sendai 980-0845, Japan\\
Miyagi University of Education\\
Department of Mathematics
\end{document}